\documentclass[11pt,reqno]{amsart}
\usepackage{amsmath}
\usepackage{color}
\usepackage{amsfonts}
\usepackage{setspace}
\usepackage[toc,page]{appendix}
 
 \newtheorem{thm}{Theorem}[section]
 \newtheorem{prop}[thm]{Proposition}
 \newtheorem{lem}[thm]{Lemma}
 
 \newtheorem{definition}[thm]{Definition}

\numberwithin{equation}{section}

\renewcommand{\Re}{\operatorname{Re}}

\newcommand{\rg}{\operatorname{rg}}

\newcommand{\R}{\mathbb{R}}
\newcommand{\la}{\lambda}

\newcommand{\C}{\mathbb{C}}
\newcommand{\Hb}{\overline{\mathbb{H}}}

\title{Mode stability of self-similar wave maps in higher dimensions}

\author{Ovidiu Costin}
\address{Department of Mathematics, The Ohio State University, 231 W 18th Ave, Columbus, OH, 43220, USA}
\email{costin@math.ohio-state.edu}
\author{Roland Donninger}
\address{Rheinische Friedrich-Wilhelms-Universit\"at Bonn,
Mathematisches Institut, Endenicher Allee 60, D-53115 Bonn, Germany}
\address{Faculty of Mathematics, University of Vienna, Oskar-Morgenstern-Platz 1, A-1090 Vienna, Austria}
\email{donninge@math.uni-bonn.de}
\thanks{O. C. and I. G. were  partially supported by the NSF  DMS grant  1515755. 
R. D.  is supported by a Sofja Kovalevskaja Award granted by 
the Alexander von Humboldt Foundation 
and the German Federal Ministry of Education and Research;
partial support by the DFG, CRC 1060, is also gratefully acknowledged. }
\author{Irfan Glogi\'c}
\address{Department of Mathematics, The Ohio State University, 231 W 18th Ave, Columbus, OH, 43220, USA}

\email{glogic.1@osu.edu}

\usepackage{kantlipsum}
\usepackage[colorlinks=true,linkcolor=red,citecolor=red,urlcolor=blue]{hyperref}

\begin{document}

\begin{abstract}
We consider co-rotational wave maps from Minkowski space in $d+1$ dimensions to the $d$-sphere. Recently, Bizo\'n and Biernat found explicit self-similar solutions for each dimension $d\geq 4$. We give a rigorous proof for the mode stability of these self-similar wave maps.
\end{abstract}

\maketitle
\section{Introduction}

Let $(M,g)$ be a Riemannian manifold with metric $g$.
Wave maps on $(d+1)$-dimensional Minkowski space $(\R^{1,d},\eta)$ arise from the geometric action principle
\[ S(u)=\int_{\R^{1,d}}\eta^{\mu\nu}(u^* g)_{\mu\nu}=\int_{\R^{1,d}}\eta^{\mu\nu}\partial_\mu u^a \partial_\nu u^b g_{ab}\circ u \]
as solutions $u: \R^{1,d}\to M$ of the corresponding Euler-Lagrange equation.
The wave maps action is a rich source for interesting nonlinear relativistic field theories that play an important role in mathematical physics, e.g.~as models for Einstein's equation or in the description of ferromagnetism. Furthermore, wave maps are prototypical examples of geometric wave equations that attracted a lot of interest from the PDE community, see e.g.~\cite{Str03, KriSchTat08, SteTat10, RodSte10, KriSch12, RapRod12, CotKenLawSch15a, CotKenLawSch15b} for some recent contributions to the large-data problem.

In this paper we restrict ourselves to the special case $M=\mathbb{S}^d$.
By choosing standard hyperspherical coordinates on $\mathbb{S}^d$ and spherical coordinates on Minkowski space, one may consider so-called co-rotational maps $u: \R^{1,d}\to \mathbb{S}^d$ which are of the form $u(t,r,\omega)=(\psi(t,r),\omega)$, where $\omega\in \mathbb{S}^{d-1}$.
Under this symmetry assumption the wave maps equation reduces to the single semilinear wave equation 
\begin{equation}\label{eq:main} \psi_{tt}-\psi_{rr}-\frac{d-1}{r}\psi_r=-\frac{(d-1)\sin(\psi)\cos(\psi)}{r^2},
\end{equation}   
see \cite{CazShaTah98}.
In the case $d=3$, Eq.~\eqref{eq:main} admits the explicit self-similar solution
\[ \psi^T(t,r)=2\arctan\left (\frac{r}{T-t}\right ), \]
where $T>0$ is the blow-up time \cite{Sha88, TurSpe90}.
Based on numerics \cite{BizChmTab00}, $\psi^T$ is conjectured to describe the generic blow-up profile. The nonlinear asymptotic stability of $\psi^T$ was rigorously proved in \cite{Don11, DonSchAic12, CosDonXia14}, see also \cite{DonSch12, Don14, DonSch14, DonSch14a, Don15} for similar results related to other equations. In the case $d\geq 4$, Bizo\'n and Biernat \cite{BizonBiernat15} discovered the explicit self-similar solution
\begin{equation}\label{eq:BBsol}
\psi^T(t,r)=2\arctan\left(\frac{1}{\sqrt{d-2}}\frac{r}{T-t}\right).
\end{equation} 
In the same paper, it is conjectured that this solution exhibits a stable blowup pattern and numerical evidence is provided in support of this conjecture. In the present paper we address the stability question for self-similar solutions and we rigorously prove the \emph{mode stability} of the Bizo\'n-Biernat solution \eqref{eq:BBsol} for all $d\geq 3$.

\section{Mode stability problem}
  
The problem of mode stability can be formulated for explicit self-similar solutions coming from a non-linear radial wave equation of the general form 
\begin{equation}\label{eq:general} 
      \psi_{tt}-\psi_{rr}-\frac{d-1}{r}\psi_r=-\frac{g(\psi)}{r^2}.
\end{equation}
Therefore, we describe it in this more general context.
\subsection{Self-similar solutions}

Equations of the type~\eqref{eq:general} have the natural scaling
\[
\psi(t,r)\mapsto \psi_\la(t,r)=\psi\left(\frac{t}{\la},\frac{r}{\la}\right),\quad \lambda>0.
\]
Therefore, self-similar solutions to~\eqref{eq:general} have the form
\begin{equation}\label{self-similar_ansatz}
\psi(t,r)=f(\rho),\quad \rho=\frac{r}{T-t},
\end{equation}
where a positive parameter $T$ is allowed due to the time translation symmetry of~\eqref{eq:general}. By inserting the ansatz~\eqref{self-similar_ansatz} into~\eqref{eq:general} we get an ordinary differential equation for $f$,
\begin{equation}\label{eq:ODEansatz}
	(1-\rho^2)f''+\left(\frac{d-1}{\rho}-2\rho\right)f'-\frac{g(f)}{\rho^2}=0.
\end{equation}
Note that 
\[
\frac{\partial^n}{\partial r^n}f\left(\frac{r}{T-t}\right)\bigg|_{r=0}=\frac{1}{(T-t)^n}f^{(n)}(0).
\]
Therefore, if a solution to equation~\eqref{eq:ODEansatz} has a non-vanishing derivative at zero, it provides an example of a solution to equation \eqref{eq:general} that develops a singularity as $t\rightarrow T^-$. Since the breakdown happens at $(T,0)$, and due to finite speed of propagation, we are interested only in solutions in the backward light-cone of the blow-up point $(T,0)$, which corresponds to $\rho\in[0,1]$.
Consequently, we restrict our attention to solutions 
of~\eqref{eq:ODEansatz} that are smooth on $[0,1]$. 

\subsection{Mode stability}
Let
\[
\psi^T(t,r)=f\left(\frac{r}{T-t}\right)
\]
be a self-similar solution (a family of solutions, to be precise) to~\eqref{eq:general}, for some $f\in C{^{\infty}[0,1]}$, that exhibits a finite time blow-up at $t=T$. Our aim is to analyze the stability of this solution.

Due to the self-similarity of $\psi^T$, it is convenient to introduce the similarity coordinates
\begin{equation}\label{eq:sim_coordinates}
	\tau=-\log(T-t)\quad \text{and} \quad \rho=\frac{r}{T-t}.
\end{equation}
Equation~\eqref{eq:general} is thereby transformed into
\begin{equation}\label{eq:main2}
\phi_{\tau\tau}+\phi_\tau+2\rho\phi_{\tau\rho}-(1-\rho^2)\phi_{\rho\rho}-\left(\frac{d-1}{\rho}-2\rho\right)\phi_\rho=-\frac{g(\phi)}{\rho^2},
\end{equation}
where $\phi(\tau,\rho)=\psi(T-e^{-\tau},\rho e^{-\tau})$. 
Note that in the coordinates above, the problem of stability of finite time blow-up $(t\rightarrow T^-)$ is transformed
into an asymptotic stability $(\tau\rightarrow\infty)$ problem. Following standard methods, we look for solutions to~\eqref{eq:main2} of the type
\begin{equation}\label{mode_ansatz} \phi(\tau,\rho)=f(\rho)+e^{\lambda\tau}u_\lambda(\rho),\qquad \lambda\in \C.
\end{equation}
By inserting the so-called \emph{mode ansatz}~\eqref{mode_ansatz} 
into equation~\eqref{eq:main2} and linearizing  in $u_{\la}$ we get a generalized eigenvalue equation 
\begin{equation}\label{eq:linODE}
 (1-\rho^2)u_\la''+\left[\frac{d-1}{\rho}-2(\la+1)\rho\right]u'_\la
-\la(\la+1)u_\la-V(\rho)u_\la=0,
\end{equation}
where 
\begin{equation}\label{eq:modePOT}
V(\rho)=\frac{g'(f(\rho))}{\rho^2}.
\end{equation}
 By admissible solutions\footnote{Strictly speaking, one needs to justify why only smooth solutions are expected to be relevant here. To this end, a suitable well-posedness theory for equation \eqref{eq:main2} is required. This issue will be addressed in a forthcoming publication. For the moment we rely on the experience with wave maps in dimension $d=3$, cf.~\cite{Don11, DonSchAic12, CosDonXia14}.} to~\eqref{eq:linODE} we mean the ones that belong to $C^\infty[0,1]$, and call them \emph{mode solutions}. Consequently, a non-zero mode solution $u_{\la}$ to ~\eqref{eq:linODE} with $\Re\la\geq 0$ is called an \emph{unstable mode}, and the corresponding $\la$ is called an (\emph{unstable}) \emph{eigenvalue}. 

As a matter of fact, due to the freedom of choice of the parameter $T$, equation~\eqref{eq:linODE} has an unstable mode solution that corresponds to $\lambda=1$.  Indeed, if we write~\eqref{eq:general} as 
\begin{equation}\label{eq:nonlinoper}
\mathcal{N}(\psi)=0,
\end{equation}
then due to the existence of the one parameter family of solutions $\psi^T$ to~\eqref{eq:nonlinoper} we have
\begin{align*}
\nonumber 0=\partial_T[\mathcal{N}(\psi^T)]=\mathcal{N}'(\psi^T)\partial_T \psi^T&=\mathcal{N}'(\psi^T)\left[-\frac{r}{(T-t)^2}f'\left(\frac{r}{T-t}\right)\right]\\
&=\mathcal{N}'(\psi^T)[-e^{\tau}\rho f'(\rho)].
\end{align*}
Hence, for $\la=1$, equation~\eqref{eq:linODE}  has the mode solution 
\begin{equation}\label{eq:u1}
	u_1(\rho)=\rho f'(\rho),
\end{equation}
which we call the \emph{symmetry mode}. For this reason, the \emph{symmetry eigenvalue} $\la=1$ does not correspond to a ``real" instability of the solution $\psi^T$, and we are therefore led to the following definition.
\begin{definition}
	The solution $\psi^T$ \emph{(}or $f$\emph{)} is said to be \emph{mode stable} if $u_1$ is the only unstable mode.
\end{definition}
\noindent We now state our main result.
\begin{thm}\label{thm:main}
	For any dimension $d\geq 3$, the Bizo\'n-Biernat solution~\eqref{eq:BBsol} is mode stable.
\end{thm}

The case $d=3$ is treated in \cite{CosDonXia14}. However the method developed in \cite{CDGH16} gives rise to a much shorter proof which we include at the end (see \S\ref{sec_d=3}). Furthermore, we remark that the spectral problem \eqref{eq:linODE} is truly non-self-adjoint, i.e., it cannot be transformed to a standard self-adjoint Sturm-Liouville problem (see \cite{CDGH16} for a discussion on this). Consequently, standard methods do not apply. Also, the method developed in \cite{CDGH16} is not directly applicable in the general case, when the dimension $d$ is not fixed. For that reason, an improvement of that method is required in order to account for the additional parameter, $d$.

\section{The supersymmetric problem}\label{SUSYproblem}
For the further analysis it is convenient to ``remove" the eigenvalue $\la=1$. More precisely, we wish to formulate a dual problem to~\eqref{eq:linODE} that contains all its unstable eigenvalues except for $\la=1$, and then prove non-existence of unstable eigenvalues for the new problem. 
We derive the dual problem by a suitable adaptation of a well-known procedure from supersymmetric quantum mechanics, which we briefly describe here.

\subsection{Interlude on SUSY quantum mechanics}
Consider the Schr\"odinger operator $H=-\partial_x^2+V$ on $L^2(\R)$
with some nice potential $V$
and suppose there exists a ground state $f_0 \in L^2(\R)\cap C^\infty(\R)$, i.e., $f_0''=Vf_0$.
Assume further that $f_0$ has no zeros.
Then one has the factorization
\[ -\partial_x^2 + V=\left (-\partial_x-\frac{f_0'}{f_0}\right )
\left (\partial_x-\frac{f_0'}{f_0} \right )=:Q^*Q. \]
By interchanging the order of this factorization, one defines the SUSY partner $\tilde H$
of $H$, i.e., $\tilde H:=QQ^*$. Explicitly, the SUSY partner is given by
\begin{align*} 
\tilde H&=\left (\partial_x-\frac{f_0'}{f_0}\right )\left (-\partial_x-\frac{f_0'}{f_0}\right )
=-\partial_x^2-V+2\frac{f_0'^2}{f_0^2}=:-\partial_x^2 + \tilde V
\end{align*}
where 
\[
\tilde V=-V+2\frac{f_0'^2}{f_0^2}
\]
is called the SUSY potential.
The point of all this is the following. Suppose $\lambda$ is an eigenvalue
of $H$, i.e., $Hf=Q^*Q f=\lambda f$ for some (nontrivial) $f$.
Applying $Q$ to this equation yields 
$QQ^* Q f=\lambda Q f$, i.e., $\tilde HQf=\lambda Qf$.
Thus, if $Qf\not=0$, i.e., if $f\notin \ker Q$, then $\lambda$ is an eigenvalue of $\tilde H$
as well.
Obviously, we have $\ker Q=\langle f_0\rangle$ and thus, if $\lambda\not=0$ is an 
eigenvalue of $H$, then it is also an eigenvalue of $\tilde H$.
Moreover, $0$ is not an eigenvalue of $\tilde H$ for if this were the case,
we would have either $QQ^* f=0$ for a nontrivial $f$, i.e., $f\in \ker Q^*$ or $Q^* f\in \ker Q$.
The former is impossible since $\ker Q^*=\langle \frac{1}{f_0}\rangle$ but $\frac{1}{f_0}
\notin L^2(\R)$.
The latter is impossible since $\rg Q^* \perp \ker Q$.
In summary, $\tilde H$ has the same set of eigenvalues as $H$ except for $0$.

\subsection{The Supersymmetric problem}

\indent We now implement a version of this procedure to derive the so-called \emph{supersymmetric problem} corresponding to~\eqref{eq:linODE}. It is convenient to write $d=2m+1$, i.e., $m$ is half-integer. Furthermore, we assume that the symmetry mode \eqref{eq:u1} has no zeros in $(0,1)$. Then, by the $s$-homotopic transformation
\begin{equation}\label{eq:transf1}
	u_{\la}(\rho)=\rho^{-m}(1-\rho^2)^{-\frac{\la+1-m}{2}}v_{\la}(\rho),
\end{equation} 
equation~\eqref{eq:linODE} is reduced to its normal form
\begin{equation}\label{eq:normalform}
-v''_{\la}+\left[\frac{V(\rho)}{1-\rho^2}+\frac{(m-1)(\rho^2+m)}{\rho^2(1-\rho^2)^2}\right]v_{\la}=\frac{\la(2m-\la)}{(1-\rho^2)^2}v_{\la}.
\end{equation}
By letting 
\begin{equation*}
V_1(\rho)=\frac{V(\rho)}{1-\rho^2}+\frac{(m-1)(\rho^2+m)}{\rho^2(1-\rho^2)^2}-\frac{2m-1}{(1-\rho^2)^2},
\end{equation*}
\eqref{eq:normalform} becomes
\begin{equation}\label{eq:normalform1}
-v''_{\la}+V_1(\rho)v_{\la}=\frac{(\la-1)(2m-1-\la)}{(1-\rho^2)^2}v_{\la},
\end{equation}
and clearly  $-v''_1+V_1(\rho)v_1=0$, where $v_1(\rho)$ is gotten from~\eqref{eq:u1} and~\eqref{eq:transf1} for $\la=1$.
For simplicity, we denote $\displaystyle{\frac{v'_1(\rho)}{v_1(\rho)}}$ by $w(\rho)$. Then,~\eqref{eq:normalform1} can be factorized as
\begin{equation*}
(-\partial_\rho-w)(\partial_\rho-w)v_\la=\frac{(\la-1)(2m-1-\la)}{(1-\rho^2)^2}v_{\la},
\end{equation*}
or
\begin{equation}\label{eq:SUSYprelim}
-(1-\rho^2)^2(\partial_\rho+w)(\partial_\rho-w)v_\la=(\la-1)(2m-1-\la)v_{\la}.
\end{equation}
By setting 
\begin{equation}\label{eq:transf2}
	\displaystyle{\tilde{v}_\la(\rho)=(\partial_\rho-w)v_\la(\rho)}
\end{equation}
and applying $\displaystyle{\partial_\rho-w}$ to both sides of ~\eqref{eq:SUSYprelim} we get
\begin{equation}\label{eq:commute}
-(\partial_\rho-w)[(1-\rho^2)^2(\partial_\rho+w)]\tilde{v}_\la=(\la-1)(2m-1-\la)\tilde{v}_{\la}.
\end{equation}
Note that 
\[
\displaystyle{v_1(\rho)=\rho^{m}(1-\rho^2)^{1-\frac{m}{2}}u_1(\rho)=\rho^{m+1}(1-\rho^2)^{1-\frac{m}{2}}f'(\rho)}.
\]
  Then~\eqref{eq:commute} becomes
\begin{equation}\label{eq:last}
-(1-\rho^2)^2\tilde{v}''_\la+4(1-\rho^2)\rho\tilde{v}'_\la+(1-\rho^2)W(\rho)\tilde{v}_\la=(\la-1)(2m-1-\la)\tilde{v}_{\la},
\end{equation}
where 
\[
W(\rho)=(1-\rho^2)[w(\rho)^2-w'(\rho)]+4\rho w(\rho).
\]
Now, after changing variables again to 
\begin{equation}\label{eq:transf3}
	\tilde{v}_\la(\rho)=\rho^m(1-\rho^2)^{\frac{\la-1-m}{2}}\tilde{u}_\la(\rho),
\end{equation}
 equation~\eqref{eq:last} is transformed into what we call the \emph{supersymmetric problem}
\begin{equation}\label{eq:SUSYODE}
 (1-\rho^2)\tilde{u}''_\la+\left[\frac{d-1}{\rho}-2(\la+1)\rho\right]\tilde{u}'_{\la}-\la(\la+1)\tilde{u}_\la-\tilde{V}(\rho)\tilde{u}_\la=0,
\end{equation}
with the \emph{supersymmetric potential}
\begin{equation}\label{eq:SUSYpot}
\displaystyle{\tilde{V}(\rho)=W(\rho)+\frac{m(\rho^2-m+1)}{\rho^2(1-\rho^2)}}-2.
\end{equation}

\subsection{Eigenvalue correspondence}
For general functions $f$ and $g$, the symmetry eigenvalue $\la=1$ is not necessarily removed by passing to the supersymmetric problem.  However, in the specific case of the non-linearity 
\[
g(\psi)=(d-1)\sin(\psi)\cos(\psi),
\]
and of the Bizo\'n-Biernat solution
\[
f(\rho)=2\arctan\left(\frac{\rho}{\sqrt{d-2}}\right)
\]
the set of unstable eigenvalues for both problems is the same except for $\la=1$.
From~\eqref{eq:modePOT} we obtain the mode potential 
\begin{equation}\label{eq:wavePOT}
V(\rho)=\frac{g'(f(\rho))}{\rho^2}=\frac{d-1}{\rho^2}\frac{\rho^4+(12-6d)\rho^2+(d-2)^2}{(\rho^2+d-2)^2}.
\end{equation}
Note that, in this case, the symmetry mode is
\begin{equation}\label{u1}
	u_1(\rho)=\rho f'(\rho)=\frac{2\rho\sqrt{d-2}}{d-2+\rho^2},
\end{equation}
which has no zeros in $(0,1)$.
Furthermore, following the general procedure in the previous section, we obtain the supersymmetric  potential
\begin{equation}\label{eq:SUSYwavePOT}
\tilde{V}(\rho)=-\frac{2(d-2)}{\rho^2}\frac{\rho^2-d}{\rho^2+d-2}.
\end{equation}
\begin{prop}\label{prop:eigenvCorr}
	If $\la\neq 1$ is an unstable eigenvalue of the problem~\eqref{eq:linODE} with $V$ given in \eqref{eq:wavePOT}, then $\la$ is also an unstable eigenvalue of the problem~\eqref{eq:SUSYODE} with $\tilde V$ given in \eqref{eq:SUSYwavePOT}.   
\end{prop}
\begin{proof}
	Let $\la\neq 1$ be an unstable eigenvalue of~\eqref{eq:linODE}. The sets of Frobenius indices of~\eqref{eq:linODE}  at $\rho=0$ and $\rho=1$ are $\{1,-2m\}$ and $\{0,m-\la\}$\footnote{Here $d=2m+1$.}  respectively. For the moment, we assume that $m-\la$ is not a non-negative integer. Since, by definition, the unstable mode $u_\la$ is in $C^\infty[0,1]$, the Frobenius theory tells us that $u_\la$ is in fact analytic in a neighborhood of $[0,1]$ and
	\begin{align*}
		&u_\la(\rho)\simeq \rho \quad \text{as} \quad \rho\rightarrow 0^+,\\
		&u_\la(\rho)\simeq 1 \quad \text{as} \quad \rho\rightarrow 1^-.	
	\end{align*}
	Note that through the transformations~\eqref{eq:transf1},~\eqref{eq:transf2} and~\eqref{eq:transf3} $u_\la$ becomes a solution $\tilde{u}_\la$ to the supersymmetric problem~\eqref{eq:SUSYODE}. It remains to prove that $\tilde{u}_\la$ is also analytic.
	From~\eqref{eq:transf1} it follows that $v_\la(\rho)\simeq \rho^{1+m}$ as $\rho\rightarrow 0^+$, and $v_\la(\rho)\simeq (1-\rho)^{\frac{\la+1-m}{2}}$ as $\rho\rightarrow 1^-.$
	Also, from~\eqref{eq:transf1} and~\eqref{u1} it follows that
	\begin{equation*}	
	\begin{array}{lll}
	&\displaystyle{\frac{v^\prime_1}{v_1}(\rho)}= (1+m)\rho^{-1} + O(1) &\text{as} \quad \rho\rightarrow 0^+,\vspace{1mm}\\
	&\displaystyle{\frac{v^\prime_1}{v_1}(\rho)}= \tfrac{m-2}{2}(1-\rho)^{-1} +O(1)
	&\text{as} \quad \rho\rightarrow 1^-.	
	\end{array}	
	\end{equation*}
	Therefore,~\eqref{eq:transf2} implies that
	$\tilde{v}_\la(\rho)=O(\rho^{1+m})$ as  $\rho\rightarrow 0^+$, and $\tilde{v}_\la(\rho)\simeq (1-\rho)^{\frac{\la-1-m}{2}}$ as $\rho\rightarrow 1^-$.	
	Finally, the transformation~\eqref{eq:transf3} gives
	\begin{align}
	&\tilde{u}_\la(\rho)=O(\rho)\quad \text{as} \quad \rho\rightarrow 0^+,\label{eq:asympt_at_0}\\
	&\tilde{u}_\la(\rho)\simeq 1 \quad\;\;\quad\text{as} \quad \rho\rightarrow 1^-.\label{eq:asympt_at_1}	
	\end{align}
	The fact that the Frobenius indices of the supersymmetric problem~\eqref{eq:SUSYODE} at $\rho=0$ are 2 and $-2m-1$ and~\eqref{eq:asympt_at_0} imply that $\tilde{u}_\la$ is analytic at $\rho=0$. On the other side, the Frobenius indices of~\eqref{eq:SUSYODE} at $\rho=1$ are 0 and $m-\la$. This together with \eqref{eq:asympt_at_1} and the fact that the expansion of $\tilde{u}_\la(\rho)$ at $\rho=1$ contains only integer powers imply that $\tilde{u}_\la$ is analytic also at $\rho=1$. 
	
	In the case when $m-\la$ is a non-negative integer, the same procedure as above gives the analyticity of $\tilde{u}_\la$ at $\rho=0$. At $\rho=1$, however, we have $u_\la(\rho)\simeq (1-\rho)^{m-\la}$ as $\rho\rightarrow 1^-$. Then the procedure above leads to
	$\tilde{u}_\la(\rho)\simeq (1-\rho)^{m-\la}$ as $\rho\rightarrow 1^-$, and the analyticity of $\tilde{u}_\la$ at $\rho=1$ follows. 
\end{proof}
\section{Absence of eigenvalues for the supersymmetric problem}
In the rest of the paper we let $d=k+2$. The supersymmetric problem~\eqref{eq:SUSYODE} with potential~\eqref{eq:SUSYwavePOT} now becomes
\begin{equation}\label{eq:SUSY}
(1-\rho^2)\tilde{u}_\lambda''+\left[\frac{k+1}{\rho}-2(\lambda+1)\rho\right]\tilde{u}'_\lambda
-\lambda(\lambda+1)\tilde{u}_\lambda+\frac{2k}{\rho^2}\frac{\rho^2-k-2}{\rho^2+k}\tilde{u}_\lambda=0.
\end{equation} 
In this section we prove
\begin{thm}\label{thm:SUSYproof}
	For any integer $k\geq2$, i.e. $d\geq4$, the supersymmetric problem~\eqref{eq:SUSY} does not have unstable eigenvalues.
\end{thm}
\noindent This theorem together with Proposition~\ref{prop:eigenvCorr} implies the main result Theorem~\ref{thm:main}. Also, this theorem implies that the symmetry eigenvalue $\la=1$ is indeed removed by passing to the supersymmetric problem.

To get a better insight into the analyticity properties of solutions to equation~\eqref{eq:SUSY} it is convenient to introduce a change of both independent and dependent variable
\begin{equation}\label{eq:change_of_var}
	x=\rho^2, \quad \tilde{u}_\la(\rho)=xy(x).
\end{equation}
\noindent This transformation brings~\eqref{eq:SUSY} to Heun's equation in its canonical form \cite{NIST}
\small
\begin{equation}\label{eq:heunform}
y''+\frac{1}{2}\left(\frac{k+6}{x}+\frac{2\la-k+1}{x-1}\right)y'+\frac{1}{4}\frac{( \la+3 )  ( \la+2) x+(k{\la}^{2}+5k\la+2k-4)}{x(x-1)(x+k)}y=0
\end{equation} \normalsize
\noindent Note that the analytic solution to~\eqref{eq:SUSY} at 0 is of the form $\tilde{u}_\la(\rho)=v(\rho^2)$ for some function $v$ analytic at 0. Therefore,~\eqref{eq:change_of_var} preserves the analyticity of solutions at 0 and 1, and consequently, equations~\eqref{eq:SUSY} and~\eqref{eq:heunform} have the same set of eigenvalues.

Although only one step in complexity beyond the hypergeometric class of special
functions, the Heun class is much more diverse and our understanding of them is still at an unsatisfactory level. In particular, the general connection problem for these equations is unresolved. Therefore, in the rest of the paper we use a different approach to connecting analytic solutions at 0 and 1. Namely, starting with the power series representation of the solution to~\eqref{eq:heunform} that is analytic at $x=0$, we determine its analyticity properties at $x=1$ from the asymptotic behavior of its Taylor coefficients. The approach in \cite{CosDonXia14} exploits the relationship between the recurrence relation for the Taylor coefficients and continued fractions, an idea which is quite old and has been used in different contexts, see e.g. \cite{J34,Leaver85,CCLR01,Bizon05}. However, in this paper we take a different route, analogous to the one in \cite{CDGH16}, and rely entirely on a carefully constructed approximate solution (quasi-solution) to the recurrence relation. We then use the quasi-solution to prove that for any $\la$ in the closed right half-plane (which we from now on denote by $\Hb$) the radius of convergence of the power series is 1. For $k\geq2$, this implies non-analyticity of the solution at $x=1$, and therefore rules out the existence of unstable eigenvalues of~\eqref{eq:heunform}.  

The Frobenius indices of equation~\eqref{eq:heunform} at $x=0$ are 0 and $-2-\frac{k}{2}$, so its normalized analytic solution at $x=0$ is given by the power series
\begin{equation}\label{power_series_at_0}
\sum_{n=0}^{\infty}a_n(\lambda,k)x^{n},\quad a_0=1. 
\end{equation}
By inserting~\eqref{power_series_at_0} into equation~\eqref{eq:heunform} we obtain a recurrence relation for the sequence of coefficients $\{a_n(\la,k)\}_{n\in\mathbb{N}_0}$
\begin{align*}
	2k(n+2)(2n+k+8)\,&a_{n+2}=\\  
	\{[(\la+2n+7)(\la+2n+2)-2n]k-4(n+2)^2\}\,&a_{n+1}+\\ 
	(\la+2n+3)(\la+2n+2)\,&a_{n},
\end{align*}
where $a_{-1}=0$ and $a_0=1$, or written differently
\begin{equation}\label{eq:rec_a}
a_{n+2}(\la,k)=A_n(\la,k)\,a_{n+1}(\la,k)+B_n(\la,k)\,a_n(\la,k),
\end{equation}
where
\[
A_n(\la,k)=\frac{k{\la}^{2}+k( 4n+9) \la+4k{n}^{2}+16nk-4{n}^{2}+14k
	-16n-16
}{2k(n+2)(2n+k+8)}
\]
and
\[
B_n(\la,k)=\frac{(\la+2n+3)(\la+2n+2)}{2k(n+2)(2n+k+8)}.
\]
We now let
\begin{equation}\label{eq:def_of_r}
	r_n(\la,k)=\frac{a_{n+1}(\la,k)}{a_n(\la,k)},
\end{equation}
 and thereby transform~\eqref{eq:rec_a} into
\begin{equation}\label{eq:rec_r}
	r_{n+1}(\la,k)=A_n(\la,k)+\frac{B_n(\la,k)}{r_n(\la,k)},
\end{equation}
with the initial condition 
\[r_0(\la,k)=\frac{a_1(\la,k)}{a_0(\la,k)}=A_{-1}(\la,k)=\frac{k\la^2+5k\la+2k-4}{2k(k+6)}. \]
Recall that non-existence of unstable eigenvalues of~\eqref{eq:heunform} (and therefore Theorem \ref{thm:SUSYproof}) is implied by $\lim_{n\rightarrow \infty}r_n(\la,k)=1$ for all $\la\in\Hb$. Indeed, if $\lim_{n\rightarrow \infty}r_n(\la,k)=1$, then~\eqref{power_series_at_0} can not be analytically extended through $x=1$. First, we show that the following dichotomy holds. 
\begin{lem}\label{lem:limits}
	\label{limit_of_rn} Given $\la\in\Hb$ and integer $k\geq2$, either
	\begin{equation}\label{eq:limit1_for_rn}
	\lim_{n\rightarrow \infty} r_n(\la,k) = 1,
	\end{equation}
	or
	\begin{equation}\label{eq:limit2_for_rn}
	\lim_{n\rightarrow \infty} r_n(\la,k) = -\frac{1}{k}.
	\end{equation}
\end{lem}
\begin{proof} 
	Since 
	\[
	\lim_{n\rightarrow \infty} A_n(\lambda,k)=\frac{k-1}{k} \quad \text{and} \quad \lim_{n\rightarrow \infty} 
	B_n(\lambda,k)=\frac{1}{k},
	\]
	the characteristic equation associated to~\eqref{eq:rec_a} is
	\begin{equation}\label{char_eq}
	t^2-\frac{k-1}{k}t-\frac{1}{k}=0.
	\end{equation}
	As the solutions to~\eqref{char_eq} ($t=1$ and $t=-1/k$) have distinct moduli, 
	by a theorem of Poincar\'{e} (see, for example, \cite{Elaydi05}, p. 343, or 
	\cite{Buslaev05}), either $a_n$ is zero eventually in $n$, 
	or $\lim_{n\rightarrow \infty} a_{n+1}(\lambda,k)/a_n(\lambda,k)$ exists and it is 
	equal to either 1 or $-1/k$. Now, for a fixed $\la$ and $k$, $a_n(\la,k)$ cannot be zero eventually in $n$, 
	since by backward induction from~\eqref{eq:rec_a} one would get $a_0=0$, 
	hence the claim follows.
\end{proof}
Our aim is to come up with an approximation $\tilde{r}_n(\la,k)$ to $r_n(\la,k)$, which is simple while being uniform in $\la$ and $k$. We then use the \emph{quasi-solution} method\footnote{The quasi-solution method allows one to determine the properties of a solution to a dynamical system by using an approximation to it (called quasi-solution) that closely emulates its dynamics. This method has been successively employed in the context of nonlinear ordinary differential equations \cite{CostinHuangSchlag12, CostinHuangTanveer14,CostinKimTanveer14} and more recently in difference equations \cite{CDGH16}.} to prove that $r_n(\la,k)$ eventually enters (and stays in) a small neighborhood of 1. Via Lemma \ref{lem:limits}, this implies the main result of this section Theorem \ref{thm:SUSYproof}. 

\noindent We found it convenient to separate our analysis into two cases, $k\geq3$ and $k=2$. 
 
\noindent\textbf{Case $k\geq3$}. As we already mentioned, the reasoning in \cite{CDGH16} does not directly apply here since we face the extra challenge posed by an additional parameter, $k$. Therefore, in order to obtain a simple enough quasi-solution a non-trivial improvement upon the process given in \cite[\S4.1]{CDGH16} is required, and we describe it in \S\ref{quasi-explanation}; the quasi-solution is
\begin{equation}\label{eq:quasisol}
\tilde{r}_n(\lambda,k)=\frac{1}{2}\frac{\lambda^2}{2n^2+(k+8)n+k+5}+\frac{2\lambda}{2n+k+6}+\frac{2n+3}{2n+k+6},
\end{equation}
which turns out to be close enough to $r_n$ for the purpose of proving~\eqref{eq:limit1_for_rn}.
\begin{lem}\label{lem:r2}
	$r_2(\la,k)$ and $(\tilde{r}_n(\la,k))^{-1}$ for $n\geq2$, are analytic in $\Hb$ as functions of $\la$.
\end{lem}
\begin{proof}
Every $r_n(\la,k)$ is a ratio of two polynomials in $\la$ of degrees $2n+2$ and $2n$, with integer valued coefficients depending on $k$. The denominator of $r_2(\la,k)$
\small
\[
6k(k+10)[{k}^{2}{l}^{4}+14{k}^{2}{l}^{3}+k(63k
-8)l^2+14k(7{k}-4)l+8(5{k}^{2}-2k+8)],
\]
\normalsize
and $\tilde{r}_n(\lambda,k)$ for $n\geq2$, considered as polynomials in $\la$,
are Hurwitz-stable i.e., all of their zeros are in the (open) left half-plane, which can be straightforwardly 
checked by, say, the Routh-Hurwitz criterion or its reformulation by Wall (see \cite{Wall45}). The conclusion follows.
\end{proof}

\noindent Now, let
\begin{equation}\label{delta_definition} 
\delta_n(\la,k)=\frac{r_n(\la,k)}{\tilde{r}_n(\la,k)}-1.
\end{equation}
Substitution of~\eqref{delta_definition} into~\eqref{eq:rec_r} 
leads to the following recurrence relation for $\delta_n$,
\begin{equation}\label{delta_recurrence} 
\delta_{n+1}=\varepsilon_n+C_n\frac{\delta_n}{1+\delta_n},
\end{equation}
where
\begin{equation}\label{epsilon_and_C} 
\varepsilon_n=\frac{A_n\tilde{r}_n+B_n}{\tilde{r}_n\tilde{r}_{n+1}}-1 \quad \text{and} 
\quad C_n=\frac{B_n}{\tilde{r}_n\tilde{r}_{n+1}}.
\end{equation}

\begin{lem}\label{lem:estimates}
	For any $k\geq3$, $n\geq2$ and $\la\in\Hb$ the following estimates hold
	\begin{align}
	&|\delta_2(\la,k)|\leq\frac{1}{2}, \label{eq:d2_estimate} \\
	&|\varepsilon_n(\la,k)|\leq\frac{5}{12}-\frac{1}{k}, \label{eq:en_estimate} \\
	&|C_n(\la,k)|\leq\frac{1}{12}+\frac{1}{k}. \label{eq:cn_estimate} 
	\end{align}
\end{lem}
\begin{proof} 
The proof is analogous to the proof of Lemma 3.4 in \cite{CDGH16}. Since the process is the same for all three quantities, we illustrate it on $C_n$ only.

Lemma \ref{lem:r2} and~\eqref{epsilon_and_C} 
imply that $C_n$ is analytic (in $\la$) in $\Hb$. 
Also, being a rational function in $\la$, $C_n$ is evidently polynomially 
bounded in $\Hb$. Hence, according to the 
Phragm\'{e}n-Lindel\"{o}f principle\footnote{We use the sectorial 
formulation of this principle, see, for example, \cite{Titchmarsh58}, p. 177.}, 
it suffices to prove that~\eqref{eq:cn_estimate} holds on the imaginary line. 
	
To that end, we first bring $C_{n+2}(\lambda,k+3)$ to the form of the ratio of two 
polynomials $P_1(n,\lambda,k)$ and $P_2(n,\lambda,k)$ (note the shift in the values of $n$ and $k$). Then, for $t$ real, $|C_{n+2}(it,k+3)|^2$ is equal 
to the quotient of two polynomials, $Q_1(n,t^2,k)=|P_1(n,it,k)|^2$ 
and $Q_2(n,t^2,k)=|P_2(n,it,k)|^2$.
	
In order to show that $|C_{n+2}(it,k+3)|\leq\frac{1}{12}+\frac{1}{k+3}$, for all
$n,k\geq0$ and $t$ real, all we need is to show 
that 
\[
|C_{n+1}(it,k+3)|^2=\frac{Q_1(n,t^2,k)}{Q_2(n,t^2,k)}\leq\left(\frac{1}{12}+\frac{1}{k+3}\right)^2,
\]
or 
equivalently 
\[
(k+15)^2 Q_2-[12(k+3)]^2Q_1\geq0.
\]
When expanded, $(k+15)^2Q_2-[12(k+3)]^2Q_1$ has manifestly positive coefficients (as a polynomial in $t$, $k$ and $n$), and 
the variable $t$ appears with even powers only. Thus,~\eqref{eq:cn_estimate} holds on 
the whole imaginary line, and the result follows.
\end{proof}

\begin{proof}[Proof of the Theorem~\ref{thm:SUSYproof} for $k\geq3$]
	From~\eqref{delta_recurrence} and Lemma \ref{lem:estimates}, a simple inductive argument in $n$ implies that 
	\begin{equation}\label{eq:delta_estimate}
	|\delta_n(\la,k)|\leq\frac{1}{2}, \quad \text{for all } n\geq2,\; k\geq3 \text{, and } \lambda\in\Hb.
	\end{equation}
	Since for any fixed $k\geq3$ and $\lambda\in\Hb$, $\lim_{n\rightarrow \infty} \tilde{r}_n(\lambda) = 1$,
	~\eqref{delta_definition} and~\eqref{eq:delta_estimate} exclude the possibility of~\eqref{eq:limit2_for_rn}. 
	Hence,~\eqref{eq:limit1_for_rn} holds in $\Hb$, and the claim follows.
\end{proof}

\noindent\textbf{Case $k=2$}. We first give the analogues of Lemmas \ref{lem:r2} and \ref{lem:estimates} in this case.
\begin{lem}\label{lem:r21}
	For $n\geq4$, $r_4(\la,2)$ and $(\tilde{r}_n(\la,2))^{-1}$ are analytic in $\la\in\Hb$.
\end{lem}
\begin{lem}\label{lem:estimates1}
	For any $n\geq4$ and $\la\in\Hb$, the following hold
	\begin{equation*}
	|\delta_4(\la,2)|\leq\frac{1}{3}, \quad|\varepsilon_n(\la,2)|\leq\frac{1}{18}, \quad
	|C_n(\la,2)|\leq\frac{11}{20}.
	\end{equation*}
\end{lem} 
\noindent The proofs are similar to those of Lemmas  \ref{lem:r2} and \ref{lem:estimates}.

With these lemmas at hand, we obtain Theorem \ref{thm:SUSYproof} for $k=2$ similarly to the case $k\geq3$.  

\section{Case $d=3$}\label{sec_d=3}

In this case ($k=1$) the eigenvalue equation~\eqref{eq:heunform} becomes
\begin{equation}\label{eq:heunform2}
y''+\left(\frac{7}{2x}+\frac{\la}{x-1}\right)y'+\frac{1}{4}\frac{( \la+3 )  ( \la+2) x+({\la}^{2}+5\la-2)}{x(x-1)(x+1)}y=0.
\end{equation}
Note that now, in addition to $x=1$, $x=-1$ is a singular point of~\eqref{eq:heunform2} which also lies on the unit circle around $x=0$. Therefore the radius of convergence of the power series of a solution analytic at $x=0$ is 1 regardless of the behavior of the solution at $x=1$. However, this can be easily remedied by a suitable M\"obius transformation of the independent variable (see \cite[\S3]{Bizon05}). Namely, 
\[
z=\frac{2x}{x+1}
\]
fixes $x=0$ and $x=1$, moves $x=-1$ to infinity and maps $x=\infty$ to $z=2$. This transformation, along with
\[
y(x)=(2-z)^{\frac{\la}{2}+1}w(z),
\]       
leads to a Heun equation
\begin{equation}\label{eq:heunform3}
	w''+\left(\frac{7}{2z}+\frac{\la}{z-1}+\frac{1}{2(z-2)}\right)w'+\frac{1}{4}\frac{( \la+4)(\la+2) z-({\la}^{2}+12\la+12)}{z(z-1)(z-2)}w=0,
\end{equation}
which is isospectral to ~\eqref{eq:heunform2} while being amenable to the method in \cite[\S3]{CDGH16}. 
We therefore only sketch the proof that~\eqref{eq:heunform3} does not have unstable eigenvalues; we use the notation of the present paper.

The Taylor coefficients of the normalized analytic solution to ~\eqref{eq:heunform3} at $z=0$ satisfy~\eqref{eq:rec_a} with 
\begin{equation*}
A_n(\la)=\frac{12n^2+(8\la+56)n+\la^2+20\la+56}{8n^2+52n+72},
\end{equation*}
\begin{equation*}
B_n(\la)=-\frac{4n^2+(4\la+12)n+\la^2+6\la+8}{8n^2+52n+72}
\end{equation*}
and the initial condition $(a_{-1}=0,\;a_{0}=1)$. 
We let $r_n(\la)=a_{n+1}(\la)/a_n(\la)$ as in~\eqref{eq:def_of_r}
. By Poincar\'e's theorem we conclude that, given $\la\in\Hb$, either $\lim_{n\rightarrow \infty} r_n(\la,k) = 1$ or $\lim_{n\rightarrow \infty} r_n(\la,k) = 1/2.$ Now, we use the recipe in \cite[\S4.1]{CDGH16} to construct a quasi-solution
\begin{equation*}
\tilde{r}_n(\lambda)=\frac{\lambda^2}{8n^2+33n+28}+\frac{5\lambda}{5n+16}+\frac{5n+6}{5n+13}.
\end{equation*}    
With $\delta_n(\la)$, $\varepsilon_n(\la)$ and $C_n(\la)$ as in ~\eqref{delta_definition} and~\eqref{epsilon_and_C} we have:

\begin{lem}\label{lem:r22}
	For $n\geq1$, $r_1(\la)$ and $(\tilde{r}_n(\la))^{-1}$ are analytic in $\la\in\Hb$.
\end{lem}

\begin{lem}\label{lem:estimates2}
	For any $n\geq1$ and $\la\in\Hb$, the following hold
	\begin{equation*}
	|\delta_1(\la)|\leq\frac{1}{3},\quad
	|\varepsilon_n(\la)|\leq\frac{1}{12},\quad
	|C_n(\la)|\leq\frac{1}{2}.
	\end{equation*}
\end{lem}
The proofs are similar to those of Lemmas  \ref{lem:r2} and \ref{lem:estimates}. From here it follows that for every $\la\in\Hb$,  $\lim_{n\rightarrow \infty} r_n(\la,k) = 1$ and the non-existence of unstable eigenvalues to~\eqref{eq:heunform3} is established.
\section{Appendix}
\subsection{Description of how to obtain the quasi-solution}\label{quasi-explanation}
Due to the fact that $r_n(\la,k)$ is a ratio of two polynomials (in $\la$) whose degrees differ by two (the numerator has larger degree), we look for approximations that are quadratic in $\la$. Then, to obtain the three coefficients (as expressions in $n$ and $k$) of such polynomial, we generate three sequences
\begin{align}\label{sequences}
	&\{r_n(0,k)\}_{n\in\mathbb{N}},\nonumber\\ &\{\tfrac{1}{2}[r_n(1,k)-r_n(-1,k)]\}_{n\in\mathbb{N}},\\ &\{\tfrac{1}{2}[r_n(1,k)+r_n(-1,k)-2\,r_n(0,k)]\}_{n\in\mathbb{N}}.\nonumber
\end{align} 
The terms of all three of these sequences are ratios of two polynomials in $k$ whose degrees differ by one (the denominator has larger degree). Therefore for all three of them we look for approximations that are reciprocals of linear polynomial in $k$, with coefficients in $n$. This is done by finding linear minimax polynomial approximations\footnote{The minimax polynomial approximation of degree 
$n$ to a continuous function $f$ on a given finite interval $[a,b]$ is defined to be the best approximation, 
among the polynomials of degree $n$, to $f$ in the uniform sense on $[a,b]$. For the proof of existence 
and uniqueness of this approximation and an algorithm to obtain it, see \cite{Phillips03}, \S 2.4.} to the reciprocals of~\eqref{sequences} and fitting 
rational functions in $n$ to the coefficients.

\bibliography{references1}

\begin{thebibliography}{10}

\bibitem{BizonBiernat15}
P.~Bizo{\'n} and P.~Biernat.
\newblock Generic self-similar blowup for equivariant wave maps and
  {Y}ang-{M}ills fields in higher dimensions.
\newblock {\em Comm. Math. Phys.}, 338(3):1443--1450, 2015.

\bibitem{Bizon05}
Piotr Bizo{\'n}.
\newblock An unusual eigenvalue problem.
\newblock {\em Acta Phys. Polon. B}, 36(1):5--15, 2005.

\bibitem{BizChmTab00}
Piotr Bizo{\'n}, Tadeusz Chmaj, and Zbis{\l}aw Tabor.
\newblock Dispersion and collapse of wave maps.
\newblock {\em Nonlinearity}, 13(4):1411--1423, 2000.

\bibitem{Buslaev05}
V.~I. Buslaev and S.~F. Buslaeva.
\newblock Poincar\'e's theorem on difference equations.
\newblock {\em Mat. Zametki}, 78(6):943--947, 2005.

\bibitem{CazShaTah98}
Thierry Cazenave, Jalal Shatah, and A.~Shadi Tahvildar-Zadeh.
\newblock Harmonic maps of the hyperbolic space and development of
  singularities in wave maps and {Y}ang-{M}ills fields.
\newblock {\em Ann. Inst. H. Poincar\'e Phys. Th\'eor.}, 68(3):315--349, 1998.

\bibitem{CCLR01}
O.~Costin, R.~D. Costin, J.~L. Lebowitz, and A.~Rokhlenko.
\newblock Evolution of a model quantum system under time periodic forcing:
  conditions for complete ionization.
\newblock {\em Comm. Math. Phys.}, 221(1):1--26, 2001.

\bibitem{CostinHuangSchlag12}
O.~Costin, M.~Huang, and W.~Schlag.
\newblock On the spectral properties of {$L_\pm$} in three dimensions.
\newblock {\em Nonlinearity}, 25(1):125--164, 2012.

\bibitem{CDGH16}
Ovidiu Costin, Roland Donninger, Irfan Glogi{\'c}, and M.~Huang.
\newblock On the {S}tability of {S}elf-{S}imilar {S}olutions to {N}onlinear
  {W}ave {E}quations.
\newblock {\em Comm. Math. Phys.}, 343(1):299--310, 2016.

\bibitem{CosDonXia14}
Ovidiu Costin, Roland Donninger, and Xiaoyue Xia.
\newblock A proof for the mode stability of a self-similar wave map.
\newblock {\em Nonlinearity}, 29(8):2451, 2016.

\bibitem{CostinHuangTanveer14}
Ovidiu Costin, Min Huang, and Saleh Tanveer.
\newblock Proof of the {D}ubrovin conjecture and analysis of the tritronqu\'ee
  solutions of {$P_I$}.
\newblock {\em Duke Math. J.}, 163(4):665--704, 2014.

\bibitem{CostinKimTanveer14}
Ovidiu Costin, Tae~Eun Kim, and Saleh Tanveer.
\newblock A quasi-solution approach to nonlinear problems---the case of the
  {B}lasius similarity solution.
\newblock {\em Fluid Dyn. Res.}, 46(3):031419, 19, 2014.

\bibitem{CotKenLawSch15a}
R.~C{\^o}te, C.~E. Kenig, A.~Lawrie, and W.~Schlag.
\newblock Characterization of large energy solutions of the equivariant wave
  map problem: {I}.
\newblock {\em Amer. J. Math.}, 137(1):139--207, 2015.

\bibitem{CotKenLawSch15b}
R.~C{\^o}te, C.~E. Kenig, A.~Lawrie, and W.~Schlag.
\newblock Characterization of large energy solutions of the equivariant wave
  map problem: {II}.
\newblock {\em Amer. J. Math.}, 137(1):209--250, 2015.

\bibitem{Don11}
Roland Donninger.
\newblock On stable self-similar blowup for equivariant wave maps.
\newblock {\em Comm. Pure Appl. Math.}, 64(8):1095--1147, 2011.

\bibitem{Don14}
Roland Donninger.
\newblock Stable self-similar blowup in energy supercritical {Y}ang-{M}ills
  theory.
\newblock {\em Math. Z.}, 278(3-4):1005--1032, 2014.

\bibitem{Don15}
Roland Donninger.
\newblock Strichartz estimates in similarity coordinates and stable blowup for
  the critical wave equation.
\newblock {\em Preprint arXiv:1509.02041}, 2015.

\bibitem{DonSch12}
Roland Donninger and Birgit Sch{\"o}rkhuber.
\newblock Stable self-similar blow up for energy subcritical wave equations.
\newblock {\em Dyn. Partial Differ. Equ.}, 9(1):63--87, 2012.

\bibitem{DonSch14}
Roland Donninger and Birgit Sch{\"o}rkhuber.
\newblock Stable blow up dynamics for energy supercritical wave equations.
\newblock {\em Trans. Amer. Math. Soc.}, 366(4):2167--2189, 2014.

\bibitem{DonSch14a}
Roland Donninger and Birgit Sch{\"o}rkhuber.
\newblock On blowup in supercritical wave equations.
\newblock {\em Comm. Math. Phys.}, 346(3):907--943, 2016.

\bibitem{DonSchAic12}
Roland Donninger, Birgit Sch{\"o}rkhuber, and Peter~C. Aichelburg.
\newblock On stable self-similar blow up for equivariant wave maps: the
  linearized problem.
\newblock {\em Ann. Henri Poincar\'e}, 13(1):103--144, 2012.

\bibitem{Elaydi05}
S.~Elaydi.
\newblock {\em An introduction to difference equations}.
\newblock Undergraduate Texts in Mathematics. Springer, New York, third
  edition, 2005.

\bibitem{J34}
George Jaff{\'e}.
\newblock Zur theorie des wasserstoffmolek{\"u}lions.
\newblock {\em Zeitschrift f{\"u}r Physik}, 87(7):535--544, 1934.

\bibitem{KriSchTat08}
J.~Krieger, W.~Schlag, and D.~Tataru.
\newblock Renormalization and blow up for charge one equivariant critical wave
  maps.
\newblock {\em Invent. Math.}, 171(3):543--615, 2008.

\bibitem{KriSch12}
Joachim Krieger and Wilhelm Schlag.
\newblock {\em Concentration compactness for critical wave maps}.
\newblock EMS Monographs in Mathematics. European Mathematical Society (EMS),
  Z\"urich, 2012.

\bibitem{Leaver85}
E.~W. Leaver.
\newblock An analytic representation for the quasi-normal modes of kerr black
  holes.
\newblock {\em Proceedings of the Royal Society of London A: Mathematical,
  Physical and Engineering Sciences}, 402(1823):285--298, 1985.

\bibitem{NIST}
F.~W.~J. Olver, D.~W. Lozier, R.~F. Boisvert, and Charles~W. Clark, editors.
\newblock {\em N{IST} handbook of mathematical functions}.
\newblock U.S. Department of Commerce, National Institute of Standards and
  Technology, Washington, DC; Cambridge University Press, Cambridge, 2010.

\bibitem{Phillips03}
G.~M. Phillips.
\newblock {\em Interpolation and approximation by polynomials}.
\newblock CMS Books in Mathematics/Ouvrages de Math\'ematiques de la SMC, 14.
  Springer-Verlag, New York, 2003.

\bibitem{RapRod12}
Pierre Rapha{\"e}l and Igor Rodnianski.
\newblock Stable blow up dynamics for the critical co-rotational wave maps and
  equivariant {Y}ang-{M}ills problems.
\newblock {\em Publ. Math. Inst. Hautes \'Etudes Sci.}, pages 1--122, 2012.

\bibitem{RodSte10}
Igor Rodnianski and Jacob Sterbenz.
\newblock On the formation of singularities in the critical {${\rm O}(3)$}
  {$\sigma$}-model.
\newblock {\em Ann. of Math. (2)}, 172(1):187--242, 2010.

\bibitem{Sha88}
Jalal Shatah.
\newblock Weak solutions and development of singularities of the {${\rm
  SU}(2)$} {$\sigma$}-model.
\newblock {\em Comm. Pure Appl. Math.}, 41(4):459--469, 1988.

\bibitem{SteTat10}
Jacob Sterbenz and Daniel Tataru.
\newblock Regularity of wave-maps in dimension {$2+1$}.
\newblock {\em Comm. Math. Phys.}, 298(1):231--264, 2010.

\bibitem{Str03}
Michael Struwe.
\newblock Equivariant wave maps in two space dimensions.
\newblock {\em Comm. Pure Appl. Math.}, 56(7):815--823, 2003.
\newblock Dedicated to the memory of J{\"u}rgen K. Moser.

\bibitem{Titchmarsh58}
E.~C. Titchmarsh.
\newblock {\em The theory of functions}.
\newblock Oxford University Press, Oxford, 1958.
\newblock Reprint of the second (1939) edition.

\bibitem{TurSpe90}
Neil Turok and David Spergel.
\newblock Global texture and the microwave background.
\newblock {\em Phys. Rev. Lett.}, 64:2736--2739, Jun 1990.

\bibitem{Wall45}
H.~S. Wall.
\newblock Polynomials whose zeros have negative real parts.
\newblock {\em Amer. Math. Monthly}, 52:308--322, 1945.

\end{thebibliography}
\bibliographystyle{plain}

\end{document}